%% file: main.tex
\documentclass[11pt,a4paper]{scrartcl}
\makeatletter
\DeclareOldFontCommand{\sl}{\normalfont\slshape}{\@nomath\sl}
\DeclareOldFontCommand{\rm}{\normalfont\rmfamily}{\mathrm}
\makeatother
\usepackage[top=2.11cm,left=2.5cm,right=2.5cm,bottom=4.34cm]{geometry}
\input{preamble}
\usepackage{mathptmx}
\usepackage{breakurl}

\newcommand{\mult}{\circ}
\newcommand{\commu}{\lvert}

\begin{document}

\title{Three characterisations of the sequential product}
\author{John van de Wetering \\
Radboud University Nijmegen, Netherlands \\
\texttt{wetering@cs.ru.nl}
}

\date{\today}


\maketitle

\begin{abstract}
\noindent It has already been established that the properties required of an abstract sequential product as introduced by Gudder and Greechie are not enough to characterise the standard sequential product $a\mult b = \sqrt{a}b\sqrt{a}$ on an operator algebra. We give three additional properties, each of which characterises the standard sequential product on either a von Neumann algebra or a Euclidean Jordan algebra. These properties are (1) invariance under application of unital order isomorphisms, (2) symmetry of the sequential product with respect to a certain inner product, and (3) preservation of invertibility of the effects. To give these characterisations we first have to study \emph{convex $\sigma$-sequential effect algebras}. We show that these objects correspond to unit intervals of spectral order unit spaces with a homogeneous positive cone.
\end{abstract}

\section{Introduction}
Measuring a state in quantum theory can change it in a non-trivial way. As a result, the order of measurements done on the same system is relevant. We will study \emph{effects}, a type of measurement that has a binary outcome denoting the success of the measurement. Supposing we have two effects $a$ and $b$, we can consider the combined effect $a\mult b$ that is successful if and only if first $a$ is measured and successful and afterwards $b$ is measured and successful. We will call this measurement the \emph{sequential product} of $a$ and $b$. In general $a\mult b\neq b\mult a$. In quantum theory, effects are represented by certain elements of a C$^*$-algebra and the sequential product is defined as $a\mult b:=\sqrt{a}b\sqrt{a}$. We have $a\mult b = a\mult b$ if and only if $a$ and $b$ commute \cite{gudder2002sequentially}.

An abstract formalism for studying spaces of effects with a sequential product was given by Gudder and Greechie in the form of a \emph{sequential effect algebra} \cite{gudder2002sequential}. This is a set equipped with a partial addition representing fuzzy conjunction, a complement representing negation, and a binary operation representing the sequential product. A sequential effect algebra has been shown to exhibit some of the same properties present in the set of quantum effects \cite{gudder2002sequential,gudder2018convex}.

The goal of this paper is twofold. First to develop the theory of a particular class of sequential effect algebras known as \emph{convex $\sigma$-effect algebras}, which we dub \emph{sequential effect spaces} to stress their similarity to vector spaces, that in addition to the aforementioned structure also allow convex combinations and countable infima of effects. Second, to use these results to give three new characterisations of the standard sequential product on a variety of operator algebras. In particular our results hold for the set of bounded operators on a Hilbert space and for \emph{Euclidean Jordan algebras}, a class of operator algebras containing for instances spaces of operators on a real or quaternionic Hilbert space \cite{jordan1933}.

It was shown in \cite{weihua2009uniqueness} that the properties required of an abstract sequential product as defined by Gudder and Greechie are not enough to characterise the standard sequential product $a\mult b= \sqrt{a}b\sqrt{a}$. There do however exist characterisations based on a related set of axioms \cite{gudder2008characterization,westerbaan2016universal}. We will add to these efforts by giving characterisations based on a variety of different ideas, including an order-theoretic property, symmetry with respect to an inner product, and an algebraic property. Informally stated:

\begin{quote}
	An abstract sequential product $a\mult b$ is the standard sequential product if any of the following additional conditions is satisfied:
	\begin{itemize}
		\item For all unital order isomorphisms $\Phi$ we have $\Phi(a\mult b) = \Phi(a)\mult\Phi(b)$.
		\item When $a$ and $b$ are invertible their product $a\mult b$ is as well and $(a\mult b)^{-1} = a^{-1}\mult b^{-1}$.
		\item The space has an inner product such that $\inn{a\mult b,c} = \inn{b,a\mult c}$.
	\end{itemize}
\end{quote}
\noindent \textbf{Note:} The setting of each of these characterisations is slightly different. For details we refer to the theorems stated in the body of the paper (\ref{theor:uniqueinvariantproduct}, \ref{theor:uniquesymmetricproduct}, and \ref{theor:uniquequadraticproduct}).

To achieve these results we will develop the theory of convex $\sigma$-sequential effect algebras. In particular we will show that they correspond to unit intervals of order unit spaces which allow spectral decompositions in terms of the sharp effects (i.e.\ effects satisfying $a\mult a =a$, like the projections in quantum theory), and that have a positive cone that is \emph{homogeneous} (see theorem \ref{theor:homogen}).

The next section will give the basic definitions concerning sequential effect algebras and will introduce the main examples. In section \ref{sec:basicprops} some basic results concerning sequential effect algebras are summarized and proved. In section \ref{sec:structureofSES} we look into the structure of such spaces in more depth. We in particular give a characterisation of commutative sequential effect space that in turn helps us to prove a spectral theorem for general noncommutative sequential effect spaces. Finally in section \ref{sec:uniqueness} we use these results to give the three new characterisations of the sequential product on von Neumann algebras and Euclidean Jordan algebras.

\section{Definitions}
We will work towards defining \emph{convex $\sigma$-sequential effect algebras}. In the following definitions it should be kept in mind that the spaces under consideration act as generalisations of the unit interval of a von Neumann algebra or Euclidean Jordan algebra.

\begin{definition}
	An \emph{effect algebra} \cite{foulis1994effect} $(E,+,1,0,\perp)$ is a set $E$ equipped with a partially defined `addition' operation $+$ such that for all $x,y,z \in E$:
	\begin{itemize}
		\item Commutativity: when $x+y$ is defined, $y+x$ is also defined and $x+y=y+x$.
		\item Associativity: when $y+z$ and $x+(y+z)$ are defined then $x+y$ and $(x+y)+z$ are defined and $x+(y+z)=(x+y)+z$.
		\item Zero: the sum $x+0$ is always defined and $x+0=x$.
		\item The zero-one law: when $x+1$ is defined, then $x=0$.
		\item Complementation: for each $x$ there is a unique $y$ such that $x+y=1$. This $y$ is denoted by $x^\perp$ or $1-x$.
	\end{itemize}
\end{definition}
\textbf{Note:} In an effect algebra $(a^\perp)^\perp = a$ and $0^\perp = 1$. An effect algebra is automatically cancellative: if $x+y=x+z$ then $y=z$. As a consequence the relation $x\leq y \iff \exists z: x+z=y$ turns the set into a poset. We have $x\leq y \iff y^\perp\leq x^\perp$ \cite{foulis1994effect}.

On an effect algebra we want to define a sequential product. This should be a binary operation that satisfies properties that would be reasonable to expect from an operation modelling the measurement of one thing after another. The properties set out by Gudder \cite{gudder2002sequential} particularly consider when effects should be \emph{compatible}, that is, when their order of measurement doesn't matter.

\begin{definition}
	A \emph{sequential effect algebra} (SEA) \cite{gudder2002sequential} $(E,+,1,0,\perp,\mult)$ is an effect algebra with an additional (total) \emph{sequential product} operation $\mult$. We denote $a\commu b$ when $a\mult b = b\mult a$ (i.e.\ when $a$ and $b$ are compatible). The sequential product is required to satisfy the following axioms.
	\begin{enumerate}
			\item $a\mult (b+c) = a\mult b + a \mult c$.
			\item $1\mult a = a$.
			\item $a\mult b = 0 \implies b\mult a =0$. In this case we call the effects \emph{orthogonal}.
			\item If $a\commu b$ then $a\commu b^\perp$ and $a\mult (b\mult c) = (a\mult b)\mult c$ for all $c$.
			\item If $c\commu a$ and $c\commu b$ then also $c\commu (a\mult b)$ and if $a+b$ is defined $c\commu (a+b)$.
	\end{enumerate}
	If $a\commu b$ for all $a$ and $b$ we call the algebra \emph{commutative}. We let $L_a:E\rightarrow E$ denote the operator $L_a(b) = a\mult b$.
\end{definition}
\noindent \textbf{Note}: The multiplication operation is in general only additive in the second argument. However, if the algebra is commutative then the operation will of course be additive in both arguments.

For our purposes we require some additional structure on the order of the effect algebras we consider.
\begin{definition}
	A \emph{$\sigma$-effect algebra} \cite{gudder2002sequential} is an effect algebra where each descending sequence of effects has an infinum, i.e.\ $a_1\geq a_2\geq \ldots$ implies that the infimum $\wedge a_i$ exists. A \emph{$\sigma$-SEA} is a $\sigma$-effect algebra that is also a sequential effect algebra such that if $a_1\geq a_2\geq \ldots$ then $b\mult \wedge a_i = \wedge (b\mult a_i)$ and if $b\commu a_i$ for all $i$ then $b\commu \wedge a_i$.
\end{definition}
\noindent\textbf{Note}: Due to the presence of the complement in an effect algebra, a $\sigma$-effect algebra will also have suprema of increasing sequences: whenever $a_1\leq a_2\leq \ldots$ there exists a supremum $\vee a_i$. In this paper we only consider suprema and infima in a $\sigma$-SEA of mutually commuting effects. The definition of a $\sigma$-SEA could therefore be weakened to only requiring the existence of suprema and infima of commuting effects without changing the results of this paper.

When considering measurements of a physical system we can also consider probabilistic combinations of the measurements, e.g.\ by flipping a coin and deciding on the basis of the outcome which measurement to perform. We model this possibility by requiring our effect algebra to have a \emph{convex} structure.

\begin{definition}
	\cite{gudder1999convex} A \emph{convex effect algebra} $E$ is an effect algebra that is also a `module' over the real unit interval $[0,1]$, i.e.\ there is an action $[0,1]\times E\rightarrow E$ denoted by $(r,a)\mapsto ra$ satisfying the following conditions for $r,s\in [0,1]$ and $a,b\in E$:
	\begin{itemize}
		\item If $r+s\leq 1$ then $ra + sa$ is defined and $(r+s)a = ra + sa$.
		\item $(rs)a = r(sa)$.
		\item $0a = 0$ and $1a = a$.
		\item If $a+b$ is defined, then $ra + rb$ is defined and $r(a+b) = ra + rb$.
	\end{itemize}
\end{definition}

\begin{example}
	Let $V$ be a von Neumann algebra and let $E=[0,1]_V$ denote its set of effects, i.e.\ the self-adjoint elements between $0$ and $1$. The unit interval $E$ is a convex effect algebra where addition of $a,b\in E$ is defined whenever $a+b\leq 1$ and it is then equal to the standard addition. The complement of $a$ is given by $a^\perp = 1-a$. It is a sequential effect algebra with the operation $a\mult b = \sqrt{a}b\sqrt{a}$. We will refer to this product as the \emph{standard sequential product} on the space. Since von Neumann algebras are bounded monotone complete it is also a $\sigma$-SEA.
\end{example}

More general examples can be found if we look further than algebras over complex numbers. As an example, let $H$ be a real, complex or quaternionic finite-dimensional Hilbert space. In the same way as above we can then define a sequential product on its space of self-adjoint bounded operators $V=B(H)^\sa$. These spaces are examples of \emph{Euclidean Jordan algebras}.

\begin{definition}
	A vector space $V$ is a \emph{Jordan algebra} if it has a bilinear commutative operation $*$ with unit such that it satisfies the \emph{Jordan identity} $a^2*(b*a)=(a^2*b)*a$ for all $a,b \in V$. We call an element \emph{positive} if it can be written as a square: $a\geq 0$ iff there is a $b$ such that $a=b*b$. If $V$ is also a Hilbert space with $\inn{a*b,c}=\inn{b,a*c}$ and where $a\geq 0$ if and only if $\inn{a,b}\geq 0$ for all $b\geq 0$, then we call $V$ a \emph{Euclidean Jordan algebra} (EJA)\footnote{In the literature, Euclidean Jordan algebras are often required to be finite-dimensional. We do not specify this extra condition as some of our results can be more elegantly stated without this requirement of finite-dimensionality.}. The unit interval of an EJA is a convex effect algebra, and since an EJA is monotone complete it is also a $\sigma$-effect algebra.
\end{definition}
The Jordan multiplication on the space $B(H)^{sa}$ is given by $a*b:=\frac{1}{2}(ab+ba)$ and the inner product is $\inn{a,b}:=\tr(ab)$. It is then straightforward to check that $B(H)^\sa$ is indeed a Euclidean Jordan algebra.

\begin{definition}
	Let $a,b\in V$ be elements of an EJA. The \emph{quadratic representation} of $a$ is the operator $Q_a:V\rightarrow V$ given by $Q_a(b):= 2a*(a*b)-a^2*b$.
\end{definition}
For $V=B(H)^{sa}$ the quadratic representation is simply $Q_a(b) = aba$, it should therefore not be too surprising that the quadratic representation allows us to define a sequential product on a general Euclidean Jordan algebra.
\begin{example}
	Let $E=[0,1]_V$ be the unit interval of an EJA $V$ equipped with the product $a\mult b:= Q_{\sqrt{a}}(b)$, then $E$ is a convex $\sigma$-SEA. We will refer to this product as the \emph{standard sequential product} on the space.
\end{example}
To the author's knowledge, Euclidean Jordan algebras have not been considered before as a setting for sequential effect algebras. A proof showing that the unit intervals of EJA's are indeed sequential effect algebras is therefore included in the appendix.
	
\section{Basic propositions}\label{sec:basicprops}
This section discusses the basic structure inherent in (convex $\sigma$-)sequential effect algebras, concerning especially linearity, sharpness and commutative subalgebras.

\begin{proposition}
	Let $E$ be a sequential effect algebra and let $a,b,c \in E$.
	\begin{itemize}
		\item $a\mult 0 = 0\mult a = 0$ and $a\mult 1=1\mult a = a$.
		\item $a\mult b \leq a$.
		\item If $a\leq b$ then $c\mult a\leq c\mult b$.
	\end{itemize}
\end{proposition}
\begin{proof}
	See \cite{gudder2002sequential}.
\end{proof}

\begin{definition}
	We call a convex effect algebra \emph{Archimedian} when $\frac{1}{2}a\leq \frac{1}{2}(b+\frac{1}{n})$ for all $n\in \N_{>0}$ implies that $a\leq b$.
\end{definition}

\begin{proposition}\label{prop:normalsum} ~
	\begin{itemize}
		\item Let $E$ be an effect algebra and suppose $\wedge a_i$ exists and that $b+a_i$ is defined for all $i$, then $b+\wedge a_i = \wedge (b+a_i)$ and similarly for suprema.
		\item A convex $\sigma$-effect algebra is Archimedian.
	\end{itemize}
\end{proposition}
\begin{proof}
	See \cite{jacobs2017distances}.
\end{proof}

\begin{definition}
	Let $C\subset V$ be a subset of a real vector space $V$. It is called a \emph{cone} if it is closed under addition and multiplication by a positive scalar. It is \emph{proper} if $C\cap (-C) = \{0\}$ and span$(C)= V$. An \emph{order unit space} $(V,C,1)$ is a vector space $V$ with proper cone $C$ which is ordered by $v\leq w \iff w-v \in C$ such that $1$ is a \emph{strong Archimedian unit}: for every $v\in V$ there exists $n\in \N$ such that $-n1 \leq v \leq n1$ and if $v\leq \frac{1}{n} 1$ for all $n$ then $v\leq 0$.
\end{definition}

\begin{proposition}\label{prop:archconvous}
	An Archimedian convex effect algebra is isomorphic to the unit interval of an order unit space.
\end{proposition}
\begin{proof}
	See \cite{jacobs2016expectation} 
\end{proof}

\noindent An order unit space has a norm defined by $\norm{v} = \inf\{r\in \R_{>0} ~;~ -r1\leq v\leq r1\}$. As a result, Archimedian convex effect algebras also have a norm given in the same way.

We introduce a different name for  `convex $\sigma$-sequential effect algebra' to stress its nature as a vector space in disguise.
\begin{definition}
	We call an effect algebra $E$ a \emph{sequential effect space} (SES) when $E$ is a convex $\sigma$-SEA. We call the order unit space $V$ such that $E\cong [0,1]_V$ its \emph{associated order unit space}.
\end{definition}
If there is no risk of confusion we will also refer to the associated order unit space of a convex $\sigma$-SEA as a sequential effect space or SES.

\begin{lemma}
	Let $E$ be a convex sequential effect algebra, then for any rational number $0\leq q\leq 1$ we have $a\mult (qb) = q(a\mult b)$.
\end{lemma}
\begin{proof}
	We first note that the $n$-fold sum of $\frac{1}{n}b$ is equal to $b$ and similarly if the $n$-fold sum of $b$ is defined then $\frac{1}{n}(b+\ldots + b) = b$. Now if we take the $n$-fold sum of $a\circ (\frac{1}{n}b)$ with itself then by the additivity in the second argument we get $a\circ b$. Multiplying both sides by $\frac{1}{n}$ then shows that $a\circ (\frac{1}{n}b) = \frac{1}{n} (a\circ b)$. By simple addition we see that the same must hold for any rational number $0\leq q\leq 1$.
\end{proof}

\begin{proposition}
	Let $E$ be a sequential effect space, then for any $\lambda\in [0,1]$ we have $a\mult (\lambda b) = \lambda (a\mult b)$.
\end{proposition}
\begin{proof}
	We know that a convex $\sigma$-SEA is the unit interval of some order unit space by proposition \ref{prop:archconvous} so that there is a norm $\norm{\cdot}$ on $E$. Let $q_i$ be an increasing sequence of rational numbers that converge to $\lambda$. Using the previous lemma we compute 
	\[\norm{\lambda (a\mult b) - a\mult (\lambda b)} = \norm{(\lambda - q_i)(a\mult b) + q_i(a\mult b) - a\mult (\lambda b)} = \norm{(\lambda-q_i)(a\mult b) - a\mult ((\lambda - q_i)b)}.\]
	Since $(\lambda - q_i)b\leq (\lambda-q_i)\norm{b}1$ and the sequential product preserves the order we get $\norm{a\mult ((\lambda - q_i)b)} \leq \norm{a}\norm{(\lambda -q_i)b} = (\lambda - q_i)\norm{a}\norm{b}$. Then $\norm{\lambda (a\mult b) - a\mult (\lambda b)} \leq (\lambda - q_i)2\norm{a}\norm{b}$. This expression obviously goes to zero so that indeed $\lambda (a\mult b) = a\mult (\lambda b)$.
\end{proof}

\begin{proposition}\label{prop:linearity}
	Let $E$ be a sequential effect space. For $a,b\in E$ and $\lambda\in [0,1]$ the following holds.
	\begin{itemize}
		\item $(\lambda a)\mult b = a\mult (\lambda b) = \lambda (a\mult b)$.
		\item If $a\commu b$ then $a\commu \lambda b$.
	\end{itemize}
\end{proposition}
\begin{proof}
	Clearly $\frac{1}{n}a\commu \frac{1}{n}a$ so that by the axioms of a SEA we get $\frac{1}{n}a\commu a$. In the same way we also see that $qa\commu a$ and $qa^\perp \commu a^\perp$ for any rational $0\leq q\leq 1$. By the axioms of a SEA we then also have $qa^\perp \commu a$ and then also $a\commu (qa+qa^\perp)=q1$ so that $(q1)\mult a = a\mult (q1) = q(a\mult 1) = qa$. Now let $\lambda\in[0,1]$ be a real number and let $q_i$ be a decreasing set of rational numbers that converges to $\lambda$, then it is straightforward to show that $\wedge q_i 1 = \lambda 1$ and because $a\commu q_i 1$ we indeed have $a\commu \wedge q_i1=\lambda 1$. As a result $(\lambda a)\mult b = (a\mult (\lambda 1))\mult b = a\mult ((\lambda 1)\mult b) = a \mult (\lambda b) = \lambda (a\mult b)$. Now if $a\commu b$ then we get $a\mult (\lambda b) = \lambda (a\mult b) = \lambda (b\mult a) = (\lambda b)\mult a$ so that indeed $a\commu \lambda b$. \footnote{In the sequel we will also consider sequential products that do not necessarily preserve infima, but are instead norm-continuous in the first argument. This proposition does then still hold since $q_ia\rightarrow \lambda a$ when $q_i\rightarrow \lambda$ and we would then have $q_i(a\mult b) = (q_ia)\mult b\rightarrow (\lambda a)\mult b$.}
\end{proof}

Because $(\lambda a)\mult b = \lambda (a\mult b)$ we can also define the sequential product for any $v> 0$ in the associated OUS of $E$ by $v\mult w:= \norm{v} (\norm{v}^{-1}v)\mult w$. In the sequel we will therefore assume that the sequential product is defined for all elements in the positive cone of the order unit space associated to a sequential effect space.

\subsection{Commutative subalgebras}
\begin{definition}
	Let $S\subseteq E$ be a subset of elements of a sequential effect algebra $E$. We define the \emph{commutant} of $S$ to be $S^\prime := \{a\in E~;~ \forall s\in S: s\commu a\}$, the set of elements in $E$ that commute with every element in $S$. Similarly the \emph{bicommutant} $S^{\prime\prime}:=(S^\prime)^\prime$ of $S$ is the set of elements in $E$ that commutes with every element in $S^\prime$.
\end{definition}

\begin{lemma}
	Let $S\subseteq E$ be a subset of elements of a sequential effect algebra $E$, then $S^\prime$ is an effect sub-algebra. If $E$ is a $\sigma$-SEA, then so is $S^\prime$ and if $E$ is a convex $\sigma$-SEA then so is $S^\prime$.
\end{lemma}
\begin{proof}
	If we have $a,b\in S^\prime$ such that $a+b$ is defined in $E$ then for all $s\in S$ we have $a\commu s$ and $b\commu s$ so that $(a+b)\commu s$ and $(a\mult b)\commu s$ as well which means that $S^\prime$ is closed under addition and multiplication. It is similarly also closed under taking the complement. Obviously $1$ and $0$ commute with every element so that $0,1\in S^\prime$ which means that is indeed an effect sub-algebra. If it is a $\sigma$-SEA then whenever $a_i\in S^\prime$ have a join in $E$, since $a_i\commu s$ for all $s\in S$ we also get $\wedge a_i\commu s$ for all $s\in S$ so that $\wedge a_i\in S^\prime$. As a consequence of proposition \ref{prop:linearity}, in a convex $\sigma$-SEA when $a\commu b$ then $\lambda a\commu b$ as well so that if $E$ is a convex $\sigma$-SEA, $S^\prime$ will be so as well.
\end{proof}

\begin{proposition}\label{prop:bicommutant}
	Let $S\subseteq E$ be a set of mutually commuting elements in a sequential effect algebra $E$, then $S^{\prime\prime}$ is a commutative sequential effect sub-algebra of $E$. If $E$ is a convex $\sigma$-SEA, then $S^{\prime\prime}$ is a convex $\sigma$-SEA as well.
\end{proposition}
\begin{proof}
	By the previous lemma we already know that $S^{\prime\prime}$ is a (convex $\sigma$-)SEA, the only thing left to prove is that it is commutative. Since $S$ consists of mutually commuting elements we see that $S\subseteq S^\prime$ so that all elements of $S^{\prime\prime}$ will commute with all elements of $S$ giving $S^{\prime\prime}\subseteq S^\prime$. As a consequence all the elements of $S^{\prime\prime}$ have to commute with all other elements of $S^{\prime\prime}$ by its defining property.
\end{proof}

\begin{definition}
	Let $a\in E$ be an element of a (convex $\sigma$-)SEA. We denote the commutative sub-algebra generated by $a$ by $C(a):=\{a\}^{\prime\prime}$.
\end{definition}
When $a\commu b$ we get $a^\perp\commu b$, $a^n\commu b$ and also $(\lambda a^n+\mu a^m)\commu b$for scalars $\lambda,\mu\in [0,1]$. As a consequence, $C(a)$ contains (among other effects) all polynomials in $a$ and $a^\perp$ that are contained in $E$.
		
\subsection{Sharpness}
Some measurements are a statistical combination of other measurements. For instance the effect $a=\frac{1}{2}b +\frac{1}{2}c$ can be interpreted as flipping a coin and based on the result measuring either $b$ or $c$. Other measurements however have no such non-trivial decompositions.

\begin{definition}
	Let $E$ be an effect algebra. We call $a\in E$ \emph{sharp} when $a\wedge a^\perp = 0$, i.e. when the following holds: if $b\leq a$ and $b\leq a^\perp$ then $b=0$.
\end{definition}

\begin{proposition}\label{prop:sharpproperties}
	\cite{gudder2002sequential} Let $p,a \in E$ in a sequential effect algebra with $p$ sharp.
	\begin{itemize}
		\item $a$ is sharp if and only if $a\mult a^\perp = 0$ if and only if $a\mult a = a$.
		\item $p\leq a$ if and only if $p\mult a = a\mult p = p$.
		\item $a\leq p$ if and only if $p\mult a = a\mult p = a$.
		\item $p\mult a = 0$ if and only if $p+a$ is defined and in this case $p+a$ is the least upper bound of $p$ and $a$. The sum $p+a$ is sharp if and only if $a$ is sharp.
		\item If $a$ is sharp as well and if it commutes with $p$ then $p\mult a$ is sharp and $a\wedge p = p\mult a$.
	\end{itemize}
\end{proposition}

When our space is the unit interval of a $B(H)$, we can consider the \emph{support projection} of an effect, the projector to the subspace where it acts non-trivially. The following definition generalises this notion to arbitrary effect algebras.

\begin{definition}
	For an element $a$ in an effect algebra we let $\ceil{a}$ denote the smallest sharp effect above $a$ and $\floor{a}$ the largest sharp effect below $a$ (when they exist).
\end{definition}

\begin{proposition}\label{prop:ceilexists}
	\cite{gudder2002sequential} In a $\sigma$-SEA, $\floor{a}$ always exists and is given by $\floor{a} = \wedge a^n$ where $a^n = a\mult a \mult\ldots \mult a$. The ceiling is given by $\ceil{a} = \floor{a^\perp}^\perp$.
\end{proposition}

\begin{lemma} \label{lem:ceilzero}
	Let $E$ be a $\sigma$-SEA and let $a,b\in E$. If $b\mult a = 0$ then $b\mult \ceil{a} = 0$.
\end{lemma}
\begin{proof}
	$b\mult a = 0$ implies that $b\mult a^\perp = b$. Now note that $a^\perp = a^\perp \mult 1 = a^\perp \mult (a+a^\perp) = a\mult a^\perp + (a^\perp)^2$ where we have used that $a\commu a^\perp$. Now since $a\mult a^\perp \leq a$ we see that $b\mult (a\mult a^\perp) \leq b\mult a = 0$ so that $b = b\mult a^\perp = b\mult (a\mult a^\perp + (a^\perp)^2) = b\mult (a^\perp)^2$. This can be extended to $b\mult (a^\perp)^n = b$ so that $b\mult \floor{a^\perp} = b\mult\wedge (a^\perp)^n = \wedge (b\mult (a^\perp)^n) = \wedge b = b$ which of course implies that $0=b\mult \floor{a^\perp}^\perp = b\mult \ceil{a}$.
\end{proof}

\section{Structure of sequential effect spaces}\label{sec:structureofSES}
In this section we will see that sequential effect spaces have some useful properties and structure. In particular, we will characterise commutative sequential effect spaces as unit intervals of continuous function spaces over \emph{basically disconnected} sets. This characterisation gives us a spectral theorem for general sequential effect spaces that in turn has several interesting consequences.

\subsection{Characterisation of commutative sequential effect spaces}

\begin{definition}
	Let $X$ be a compact Haussdorff space. We call $V\subseteq X$ a \emph{zero-set} if a continuous function $f:X\rightarrow \R$ exists such that $V=f^{-1}(0)$. Similarly we call $U\subseteq X$ a \emph{cozero-set} when $U=f^{-1}((0,\infty))$ for some continuous $f$. We call $X$ \emph{basically disconnected} when the interior of all zero-sets is closed or, equivalently, when the closure of a cozero-set is again open. We will refer to these sets that are closed and open as \emph{clopen} sets. We denote by $C(X)$ the collection of continuous maps $f:X\rightarrow \R$.
\end{definition}

Basically disconnected spaces have some known properties the proofs of which we relegate to the appendix for completeness.

\begin{proposition}\label{prop:dedekindcomplete}
	If $X$ is basically disconnected, then $C(X)$ is \emph{$\sigma$-Dedekind complete}. That is, any countable bounded set of continuous functions has a supremum.
\end{proposition}
\begin{proof}
	See appendix.
\end{proof}

\begin{proposition}
	Let $X$ be a basically disconnected space, then $\{f:X\rightarrow [0,1]~;~f \text{ continuous}\}$ is a commutative convex $\sigma$-SEA with the sequential product being pointwise multiplication: $f\mult g = fg$.
\end{proposition}
\begin{proof}
	It is obviously a convex effect algebra with the normal additive structure, and the pointwise multiplication gives it a commutative sequential structure. Since pointwise multiplication is normal it is indeed a commutative convex $\sigma$-SEA.
\end{proof}

We will spend the rest of this section proving the converse: If $E$ is a commutative sequential effect space, then it is the unit interval of a $C(X)$ where $X$ is a basically disconnected space.

A convex $\sigma$-sequential effect algebra $E$ is isomorphic to the unit interval of some order unit space $V$ (proposition \ref{prop:archconvous}). As a consequence of proposition \ref{prop:linearity} the sequential product acts linearly on $V$ in the second argument. If the multiplication is commutative it will of course also be linear in the first argument, making it a bilinear operation on $V$. Of course if we have $v,w\geq 0$ in $V$ then $v\mult w\geq 0$. As a result, we are in a position to use a classic result by Kadison:

\begin{proposition}
	\cite{kadison1951representation} Let $V$ be an order unit space with a bilinear operation $\mult$ such that $1\mult v = v$ and $v\mult w\geq 0$ whenever $v,w\geq 0$, then there exists a compact Haussdorff space $X$ and an isometric embedding $\Phi: V\rightarrow C(X)$ such that $\Phi(V)$ lies dense in $C(X)$ and $\Phi(v\mult w) = \Phi(v)\Phi(w)$. If $V$ is complete in its norm then $V$ is isomorphic as an ordered algebra to $C(X)$.
\end{proposition}

\begin{proposition}
	\cite{jacobs2017distances} Let $V$ be an order unit space that is $\sigma$-Dedekind complete (i.e.\ the supremum $\wedge v_i$ of a norm bounded sequence $v_1\geq v_2\geq \ldots \geq 0$ exists), then $V$ is complete with respect to its norm.
\end{proposition}

\begin{corollary}\label{cor:commkadison}
	Let $E$ be a commutative sequential effect space, then there exists some compact Haussdorff space $X$ such that $E$ is isomorphic as a sequential effect algebra to the unit interval of $C(X)$.
\end{corollary}
\begin{proof}
	$E$ is isomorphic to the unit interval of some order unit space $V$. Since any bounded sequence of positive elements in $V$ can be rescaled to fit inside its unit interval, we see that $V$ must be $\sigma$-Dedekind complete, because $E$ is a $\sigma$-effect algebra. The claim then follows by the previous two propositions.
\end{proof}

\begin{proposition}\label{prop:sharpclopen}
	A $p\in C(X)$ is sharp iff it is the characteristic function of a clopen set.
\end{proposition}
\begin{proof}
	Let $p:X\rightarrow [0,1]$ be an effect. It is sharp if and only if it is idempotent: $\forall x\in X: p(x)^2 = p(x)$. This implies that $p(x)=0$ or $p(x)=1$. This means that $p = \chi_S$, the characteristic function of some $S\subseteq X$. Let $U\subseteq [0,1]$ be any open subset of the unit interval. If $p$ is to be continuous we must have $p^{-1}(U)$ be open again. This set is either $\emptyset, X, S, X\backslash S$ depending on whether $U$ contains $1$ or $0$ or both. We conclude that both $S$ and $X\backslash S$ must be open, so that $S$ is indeed clopen.
\end{proof}

\begin{proposition}\label{prop:imclopen}
	Let $C(X)$ be a SES. For an effect $p:X\rightarrow [0,1]$ its ceiling is $\ceil{p}=\chi_S$ where $S=\cl{p^{-1}((0,1])}$ (where $\cl{A}$ denotes the closure of a set $A$).
\end{proposition}
\begin{proof}
	Fix an effect $p:X\rightarrow [0,1]$. Write im$(p) = p^{-1}((0,1])$ for the open set where $p$ is nonzero. By proposition \ref{prop:ceilexists} $p$ has a ceiling $\ceil{p}$. By proposition \ref{prop:sharpclopen} this must be a characteristic function $\chi_S$ for some clopen $S$. Because $p\leq \chi_S$ we must have im$(p)\subseteq S$ so that $\overline{\text{im}(p)}\subseteq \overline{S}=S$. Let $T = S/ \overline{\text{im}(p)}$ which is an open set, since $S$ is open and we subtract a closed set. Towards contradiction, assume it is non-empty and pick an $x\in T$. Since $X/ T$ and $\{x\}$ are disjoint closed sets and a compact Haussdorff space is normal we can then find a continuous $f:X\rightarrow [0,1]$ such that $f(x)=1$ and $f(X/ T)=\{0\}$. In other words, $f$ is an effect with support on $T$. Then by construction  $fp = 0$. By lemma \ref{lem:ceilzero} $f\chi_S = f\ceil{p}=0$. By construction of $f$ this is not true so we arrive at a contradiction. We conclude that $T$ must have been empty and thus that $S=\overline{\text{im}(p)}$ so that $\overline{\text{im}(p)}$ is clopen.
\end{proof}
\begin{corollary}\label{cor:zerosetint}
	If $C(X)$ is a SES, then the interior of all zero-sets, that is $f^{-1}(0)$ for some continuous function $f:X\rightarrow \R$, is clopen.
\end{corollary}
\begin{proof}
	Let $C=f^{-1}(0)$ be the zero-set of some continuous function. By post-composing $f$ with the absolute value function and rescaling we can take $f$ to have its image in $[0,1]$ without changing it's zero-set so that $C$ is the kernel of some effect. The set $X\backslash C$ is then the image of an effect so that by the previous proposition $\cl{X\backslash C} = X\backslash C^\circ$ is clopen. Its complement $C^\circ$ is then clopen as well.
\end{proof}

\begin{theorem}\label{theor:commbasicallydisconnect}
	Any commutative sequential effect space is isomorphic to the unit interval of some $C(X)$ where $X$ is basically disconnected. Conversely the unit interval of any such space is a commutative sequential effect space.
\end{theorem}
\begin{proof}
	A direct consequence of the previous corollary in combination with corollary \ref{cor:commkadison}.
\end{proof}

\subsection{A spectral theorem}

Having established that commutative sequential effect spaces correspond to function spaces over a basically disconnected space, we can use this characterisation to prove the following proposition.
\begin{proposition}\label{prop:commutativespectral}
	Let $E$ be a commutative SES, then any $a\in E$ is the supremum and norm limit of an increasing sequence of effects of the form $\sum_{k=1}^n \lambda_i p_i$ where $\lambda_i>0$ and the $p_i$ are orthogonal sharp effects.
\end{proposition}
\begin{proof}
	It is true for $C(X)$ with $X$ basically disconnected (see the appendix for details), so by theorem \ref{theor:commbasicallydisconnect} it is true for $E$.
\end{proof}

This result for commutative sequential effect spaces in turns holds for arbitrary sequential effect spaces.

\begin{theorem}\label{theor:spectraltheorem}
	Let $E$ be a SES and let $a\in E$, then $a$ can be written as the supremum and norm limit of an increasing sequence of effects of the form $\sum_i \lambda_i p_i$ where $\lambda_i>0$ and the $p_i$ are orthogonal sharp effects that commute with $a$.
\end{theorem}
\begin{proof}
	$C(a)$ is a commutative SES by proposition \ref{prop:bicommutant}. The claim then follows by application of proposition \ref{prop:commutativespectral}.
\end{proof}
\begin{corollary}
	Let $E$ be a SES and let $V$ be its associated order unit space. Let $v\in V$ be an arbitrary element, then $v$ can be written as the norm limit of elements of the form $\sum_i \lambda_i p_i$ where $\lambda_i\in \R$ and $p_i$ are orthogonal sharp effects.
\end{corollary}
\begin{proof}
	If $v\geq 0$ we can rescale it such that $\norm{v}\leq 1$ and use the previous proposition. Otherwise we know that $v+\norm{v}1\geq 0$ so that we get $v+\norm{v}1 = \lim q_n$ which we rewrite to $v = \lim (q_n - \norm{v}1)$. Each $q_n$ is of the form $q_n = \sum_i \lambda_{n,i} p_{n,i}$ so we can write $q_n - \norm{v}1 = \sum_i (\lambda_{n,i} - \norm{v})p_{n,i} - \norm{v} (1-\sum_i p_{n,i})$.
\end{proof}
\textbf{Note:} Gudder also studied convex sequential effect algebras in \cite{gudder2018convex}. There the additional condition that the space is spectral was required. This result shows that in our setting we get this property for free.

It will be useful for future reference to have a name for elements of the form $\sum_i \lambda_i p_i$.
\begin{definition}\label{def:simpleeffect}
	Let $q$ be an effect in a SES $E$. We call $q$ $\emph{simple}$ when  $q=\sum_{i=1}^n \lambda_i p_i$ for some $n\in \N$ with all $\lambda_i>0$ and $p_i\neq 0$ sharp and orthogonal. We denote the subset of simple effects of $E$ by $E_0:=\{q\in E~; q \text{ simple}\}$ and similarly $V_0\subseteq V$ is the order unit space spanned by $E_0$.
\end{definition}
Theorem \ref{theor:spectraltheorem} shows that the set of simple effects $E_0$ is dense in $E$ (with respect to the topology induced by the norm) and similarly that $V_0$ is dense in $V$. In the sequel we will frequently use this fact.

Since all elements in an order unit space $V$ associated to a SES $E$ belong to a commutative subalgebra we have a functional calculus on this space: for any continuous function $f:\R\rightarrow\R$ and $a\in V$ we can define $f(a)$. In particular we can define a square root $\sqrt{a}$ for positive $a$ such that $\sqrt{a}\mult \sqrt{a} = a$. We also get some more non-trivial consequences.

\begin{definition}
	Let $V$ be an order unit space with positive cone $C$. We call $C$ a \emph{homogeneous} cone when for every two elements $a$ and $b$ in the interior of the cone $C$ there exists an order isomorphism $\Phi:V\rightarrow V$ such that $\Phi(a)=b$.
\end{definition}

\begin{theorem}\label{theor:homogen}
	Let $E$ be a SES and let $V$ be its associated order unit space. The positive cone of $V$ is homogeneous.
\end{theorem}
\begin{proof}
	An element $a\in V$ lies in the interior of the positive cone if and only if $\exists \epsilon>0$ such that $\epsilon 1\leq a$. Looking at the commutative subalgebra generated by $a$ (which contains $\epsilon 1$) we see that as a function $a:X\rightarrow \R$ we must have $a(x)\geq \epsilon$ for all $x\in X$. We can then define $f:[\epsilon, \norm{a}]\rightarrow [\norm{a}^{-1},\epsilon^{-1}]$ by $f(x)=x^{-1}$. This is a continuous function, so by application of the functional calculus there must exist a $f(a)\in V$. By construction $f(a)$ commutes with $a$ and we have $f(a)\mult a = a\mult f(a) = 1$ so that $f(a)=a^{-1}$. The multiplication operator $L_a: V\rightarrow V$ given by $L_a(b) := a\mult b$ is such that $(L_{a^{-1}}\circ L_a)(b) = a^{-1}\mult(a\mult b) = (a^{-1}\mult a)\mult b = 1\mult b = b$ so that $L_{a^{-1}} = L_a^{-1}$. Because both $L_a$ and $L_{a^{-1}}$ are positive maps, $L_a$ is an order isomorphism. For any other $b\in V$ in the interior of the positive cone (which will then also be invertible) we can then define the order isomorphism $\Phi = L_bL_{a^{-1}}$ which gives $\Phi(a) = b$.
\end{proof}

We will use this result about the homogeneity of the positive cone when we consider sequential effect spaces that are also Hilbert spaces in section \ref{sec:seshilb}.

\subsection{Sequential order}
In this section we will look at effects for which the sequential product map preserves the order on the effect algebra in both directions. The existence of these effects has important consequences for the uniqueness of the sequential product on the space.

\begin{definition}\label{def:sequentialorder}
	\cite{gudder2005uniqueness} Let $E$ be a sequential effect algebra. We call an effect $b\in E$ \emph{order preserving} when 
	\begin{itemize}
		\item for all $a\leq b$ there exists a $c\in E$ such that $b\mult c = a$,
		\item when $b\mult a\leq b\mult c$ then $\ceil{b}\mult a \leq \ceil{b}\mult c$.
	\end{itemize} 
	If all elements in $E$ are order preserving we call $E$ \emph{sequentially ordered}.
\end{definition}
\noindent \textbf{Note:} The effect $c$ such that we have $b\mult c =a$ for an $a\leq b$ is unique when we require additionally that $c\leq \ceil{b}$ \cite{gudder2005uniqueness}.

\begin{definition}
	Let $E$ be a SES with associated OUS $V$. We say an effect $b\in E$ has a \emph{pseudo-inverse} when there exists a $c\geq 0$ in $V$ such that $b\mult c = c\mult b = \ceil{b}$. If $b$ has a pseudo-inverse and $\ceil{b}=1$ then we call $b$ \emph{invertible}. We denote the pseudo-inverse of $b$ by $b^{-1}$ just like it were a regular inverse (like with order preserving effects, the pseudo-inverse is unique given that it is below $\ceil{b}$).
\end{definition}
\begin{proposition}
	An effect with a pseudo-inverse is order preserving.
\end{proposition}
\begin{proof}
	Let $b$ have a pseudo-inverse denoted by $b^{-1}$, then when $a\leq b$ we can let $c=b^{-1}\mult a \leq b^{-1}\mult b = \ceil{b}$ so that $b\mult c = b\mult (b^{-1}\mult a) = (b\mult b^{-1})\mult a = \ceil{b}\mult a = a$. Similarly when $b\mult a \leq b\mult c$ we can multiply both sides on the left with $b^{-1}$ to get $\ceil{b}\mult a \leq \ceil{b}\mult c$.
\end{proof}

\begin{proposition} \label{prop:finiteranksequentialorder}
	Every element in $E_0$ has a pseudo-inverse. In particular if $E=E_0$, then $E$ will be sequentially ordered.
\end{proposition}
\begin{proof}
	For a simple effect $b=\sum_i \lambda_i p_i$ with $\lambda_i>0$ and $p_i$ sharp orthogonal effects we let $b^{-1} = \sum_i \lambda_i^{-1} p_i$. It is then easy to verify that $b\mult b^{-1} = b^{-1}\mult b = \sum_i p_i = \ceil{b}$.
\end{proof}
	
It is not clear whether a sequential effect space is always sequentially ordered. That the sequential product in a von Neumann algebra is sequentially ordered can be demonstrated by first showing the existence of \emph{approximate pseudo-inverses} \cite{bramthesis} and then using the space's completeness in the strong operator topology. While we can show that approximate pseudo-inverses exist in a sequential effect space, we have no analogue to the strong operator topology. By proposition \ref{prop:finiteranksequentialorder} the set of pseudo-invertible elements in a SES lies dense so that in any case the sequential order property holds for almost all elements. For our purposes this is sufficient.

\begin{definition}
	\cite{cho2015quotient} Let $V$ be an order unit space and let $0\leq q\leq 1$ be an effect in $V$. A \emph{quotient} for $q$ is an order unit space $V_q$ with a positive linear map $\xi_q: V_q\rightarrow V$ that is \emph{initial} given the property $\xi_q(1)\leq q$, i.e.\ for any positive linear map $f:W\rightarrow V$ such that $f(1)\leq q$ there exists a unique positive sub-unital map $\cl{f}:W\rightarrow V_q$ such that $\xi_q\circ \cl{f} = f$.
\end{definition}

\begin{proposition}\label{prop:quotient}
	Let $E$ be a SES and let $V$ be the associated order unit space. Let $q\in E$ be an order preserving effect and let $V_q$ be the order ideal of $V$ generated by $\ceil{q}$: $V_q:=\{v\in V~;~ \exists n\in \N: -n\ceil{q}\leq v\leq n\ceil{q}\}$. Let $\xi_q: V_q\rightarrow V$ be defined by $\xi_q(a) = q\mult a$, then $\xi_q$ is a quotient for $q$.
\end{proposition}
\begin{proof}
	Of course $\xi_q(1) = q\leq q$ so that it satisfies the required condition. Now let $f:W\rightarrow V$ be a positive map such that $f(1)\leq q$ so that in particular we have $f(a)\leq q$ when $0\leq a\leq 1$. By assumption $q$ is order preserving so that there exists a unique $c_a\in E$ below $\ceil{q}$ such that $q\mult c_a = f(a)$. It is straightforward to verify that these $c$'s are `additive': $c_{a+b}=c_a+c_b$ and also that $c_{\lambda a} = \lambda c_a$. Define $\cl{f}: W\rightarrow V_q$ by $\cl{f}(a)= c_a$.  This is a positive linear sub-unital map and we have $\xi_q(\cl{f}(a)) = q\mult c_a = f(a)$ so that $\xi_q\circ \cl{f} = f$. Furthermore $\cl{f}$ is unique with this property due to the uniqueness of the $c_a$'s.
\end{proof}

\textbf{Note:} When $\xi_q$ is a quotient for $q$ and $\Theta$ is a unital order isomorphism, the map $\xi_q\circ \Theta$ is also a quotient for $q$. Because of the universal property of a quotient, \emph{any} two quotients for $q$ are related by a unital order isomorphism in this manner \cite{basthesis}. This observation is central to our characterisations of the sequential products.

\subsection{Sequential effect spaces with inner product}\label{sec:seshilb}
Sequential effect spaces seem to be closely related to Jordan algebras. In fact, the author isn't aware of any example of a sequential effect space that is not a Jordan algebra. In this section we will show that the addition on an inner product to a sequential effect space forces it to be a Euclidean Jordan algebra.

\begin{definition}\label{def:HilbertSES}
	We call a SES $E$ with associated order unit space $V$ a \emph{Hilbert} SES when it has a positive definite inner product $\inn{\cdot,\cdot}$ such that $V$ is complete with respect to the norm induced by this inner product (i.e.\ $V$ is a Hilbert space) and such that the sequential product is \emph{symmetric}: for all $a,b,c\in E$ we have $\inn{a\mult b, c} = \inn{b,a\mult c}$.
\end{definition}

\noindent Of course Euclidean Jordan algebra are examples of a Hilbert SES. The converse is true as well.

\begin{proposition}
	Let $E$ be a Hilbert SES, then the inner product of two positive elements is positive.
\end{proposition}
\begin{proof}
	Let $a$ and $b$ be positive elements, then $a\mult b$ is also positive, so by application of the spectral theorem there exists a positive $\sqrt{a\mult b}$ such that $\sqrt{a\mult b}^2 = a\mult b$. Then $\inn{a,b} = \inn{1,a\mult b} = \inn{1,\sqrt{a\mult b}^2} = \inn{\sqrt{a\mult b},\sqrt{a\mult b}}\geq 0$.
\end{proof}
\begin{corollary}
	Let $E$ be a Hilbert SES, then the positive cone is self-dual with respect to the inner product: $a\geq 0 \iff \inn{a,b}\geq 0$ for all $b\geq 0$.
\end{corollary}
\begin{proof}
	The forward implication follows by the previous proposition. If $a\mult b = 0$ then $\inn{a,b} = \inn{1,a\mult b} =\inn{1,0}=0$ so that orthogonal effects are also orthogonal with respect to the inner product. We can write any element $a$ of $V$ as $a=a^+-a^-$ where $a^+$ and $a^-$ are orthogonal positive elements and $a^-=0$ if and only if $a\geq 0$. Suppose $\inn{a,b}\geq 0$ for all $b\geq 0$, then by taking $b=a^-$ we see that we indeed must have $a^-=0$.
\end{proof}

\begin{theorem}\label{theor:hilbertisEJA}
	An order unit space $V$ is a Euclidean Jordan algebra if and only if its unit interval is a Hilbert SES.
\end{theorem}
\begin{proof}
	The unit interval of an EJA is a Hilbert SES so we only need to show the other direction. By theorem \ref{theor:homogen} the space is homogeneous, and by the previous corollary the space is self-dual with respect to the inner product. Chu proved in \cite{chu2017infinite} a generalisation of the Koecher–Vinberg theorem that is also valid in infinite-dimension, i.e.\ that an ordered Hilbert space that is self-dual with respect to an inner product and is homogeneous is a Euclidean Jordan algebra.
\end{proof}

\noindent We will see in theorem \ref{theor:uniquesymmetricproduct} that in a Hilbert SES the sequential product has to agree with the standard sequential product of a Euclidean Jordan algebra.

\section{Uniqueness of the sequential product}\label{sec:uniqueness}
A sequential product on a sequential effect space is in general not unique. For instance, it is not unique on the space of effects of a Hilbert space since for any $t\in \R$ the map $b\mapsto \sqrt{a}a^{ti}ba^{-ti}\sqrt{a}$ defines a sequential product \cite{weihua2009uniqueness}. We can however show that different sequential products have to be very closely related. In this section we will establish a few additional properties that characterise the standard sequential product on those spaces. For the duration of the section we let $\mult$ and $\mult^\prime$ be two sequential products on a given SES $E$ with associated order unit space $V$. We denote their multiplication operators by respectively $L_a(b)=a\mult b$ and $L_a^\prime(b)=a\mult^\prime b$.

\begin{definition}
	 We call a sequential product $\mult$ \emph{normal} when $a\mult \wedge b_i = \wedge (a\mult b_i)$ and $a\commu b_i\implies a\commu \wedge b_i$, i.e. when it satisfies the conditions postulated in the definition of a $\sigma$-SEA.
\end{definition}
By definition any sequential effect space allows a sequential product that is normal, but in general such a space might also allow other sequential products that are not normal. For sequential products the following propositions are only proved with respect to simple effects, while with the additional assumption of normality the conclusions extend to arbitrary effects.

\begin{proposition}\label{prop:commutingeffects}
	Let $a,b \in E_0$ be simple. If $a\mult b = b\mult a$ then $a\mult^\prime b = a\mult b$.
	 In particular if elements commute in terms of one sequential product, they commute in terms of the other. If $\mult$ and $\mult^\prime$ are normal, then this property holds for all $a,b\in E$.
\end{proposition}
\begin{proof}
	First note that when $a$ and $b$ are commuting sharp effects $a\mult b = a\wedge b = a\mult^\prime b$ by proposition \ref{prop:sharpproperties}. Therefore when $a$ and $b$ are of the form $a=\sum_i \lambda_i p_i$ and $b=\sum_j \mu_j q_j$ where all the $p_i$ and $q_j$ are commuting sharp effects we see that $a\mult b = b\mult a$ and $a\mult^\prime b = b\mult^\prime a = a\mult b$. 

	Now suppose $\mult$ and $\mult^\prime$ are normal. Let $b\in E$ and $a\in E_0$ such that $a\mult^\prime b = b\mult^\prime a$. We can write $b$ as the supremum of elements $b_n\in E_0$ such that also $b_n\mult^\prime a = a\mult^\prime b_n$ and by the previous point we then also have $b_n\mult a = a\mult b_n = a\mult^\prime b_n$. By normality $a\mult^\prime b = a\mult^\prime \wedge b_n = \wedge(a\mult^\prime b_n) = \wedge(a\mult b_n)=a\mult \wedge b_n = a\mult b$ and since $a\commu b_n$ also $a\commu \wedge b_n =b$. The argument can now be repeated but with $a\in E$ instead of $E_0$.
\end{proof}

\begin{proposition}\label{prop:relatedsequentialproducts}
	Let $\mult$ and $\mult^\prime$ be sequential products on a SES $E$ and let $q\in E$ be an order preserving effect with respect to both products. Let $V_q$ denote the order-ideal generated by $\ceil{q}$. There exists a unital order isomorphism $\Theta_q: V_q\rightarrow V_q$ such that $q\mult^\prime a = q\mult (\Theta_q(a))$ for all $a\leq \ceil{q}$.
\end{proposition}
\begin{proof}
	The product map of an order preserving element is a quotient by proposition \ref{prop:quotient}. By the universal property of the quotient the required unital order isomorphism must then exist.
\end{proof}
Since the inverse of an effect commutes with the effect itself, we see that by proposition \ref{prop:commutingeffects} inverses are independent of the chosen sequential product: $q^{-1}\mult q = q^{-1}\mult^\prime q = 1$ (if the products are normal, otherwise this only holds for simple effects). Since invertible effects are order preserving there exists a unital order isomorphism $\Theta_q:V\rightarrow V$ for any invertible $q$ such that $q\mult^\prime a = q\mult (\Theta_q(a))$ for all $a\in V$. Written in terms of the multiplication operators: $L_q^\prime = L_q\Theta_q$

For the rest of the section we will let $\Theta_q$ denote the order isomorphism on $V_q$ of a order preserving $q$ so that $q\mult^\prime a= q\mult (\Theta_q(a))$ for $a\leq \ceil{q}$.

\begin{proposition}
	Let $0< \lambda \leq 1$, then $\Theta_{\lambda q} = \Theta_q$.
\end{proposition}
\begin{proof}
	We have $\lambda (q\mult (\Theta_q(a)))=\lambda (q\mult^\prime a) = (\lambda q)\mult^\prime a = (\lambda q)\mult(\Theta_{\lambda q}(a)) = \lambda (q\mult(\Theta_{\lambda q}(a)))$. Multiplying by $\lambda^{-1}$ and the assumption that $q$ is order preserving gives $\ceil{q}\mult (\Theta_q(a)) = \ceil{q}\mult(\Theta_{\lambda q}(a))$ for all $a\leq \ceil{q}$. Since $\ceil{q}$ acts as the identity on $V_q$ we conclude that $\Theta_q = \Theta_{\lambda q}$.
\end{proof}
Since the scalar factor doesn't matter we can define a $\Theta_q$ for any positive order preserving $q$ regardless of normalisation by setting $\Theta_q:= \Theta_{\norm{q}^{-1}q}$.

\begin{proposition}\label{prop:automorphismcommute}
	For $a,b\in E_0$ commuting and $b\leq \ceil{a}$, we get $\Theta_a(b) = b$. If the sequential products are normal this holds for all commuting $a$ and $b$ when $a$ is order preserving.
\end{proposition}
\begin{proof}
	Since $a$ is in $E_0$ it is order preserving (proposition \ref{prop:finiteranksequentialorder}). By proposition \ref{prop:commutingeffects} $a\mult b = a\mult^\prime b = a\mult(\Theta_a(b))$. Because $a$ is order preserving we then get $\ceil{a}\mult b = \ceil{a}\mult(\Theta_a(b))$. Since $b$ and $\Theta_a(b)$ are below $\ceil{a}$ we conclude that $b = \Theta_a(b)$.
\end{proof}

We would like to show that for $q$ with $\ceil{q}=1$ we get $L_q\Theta_q=\Theta_qL_q$ as this would imply that $\Theta_{a\mult b} = \Theta_a\Theta_b$ for commuting $a$ and $b$ (see proposition \ref{prop:isomorphismgroup}). Unfortunately it is not clear whether this property is true for every pair of sequential products on a SES. We can however prove this for an important class of sequential effect spaces.

\begin{definition}
	Let $E$ be a SES and let $V$ be its associated order unit space. A sequential product $\mult$ on $E$ is called \emph{invariant} if for all unital order isomorphisms $\Phi: V\rightarrow V$ we have $\Phi(a\mult b) = \Phi(a)\mult\Phi(b)$. Written in terms of the multiplication operators: $\Phi L_a = L_{\Phi(a)} \Phi$.
\end{definition}
On a von Neumann algebra or a Euclidean Jordan algebra, the unital order isomorphisms are precisely the Jordan isomorphisms (bijective linear maps preserving the Jordan product). Since the standard sequential product on these spaces can be defined in terms of the Jordan product we see that the standard sequential product is indeed invariant. It is currently an open question whether every SES allows an invariant sequential product.

\begin{proposition}\label{prop:isomorphismgroup}
	Let $E$ be a SES with an invariant sequential product $\mult$ and let $\mult^\prime$ be any other sequential product (not necessarily invariant) and write $L_a^\prime = L_a\Theta_a$.  Let $a,b\in E_0$ be simple commuting invertible effects (or if $\mult$ and $\mult^\prime$ are normal, any invertible commuting $a$ and $b$).
	\begin{enumerate}
		\item $\Theta_b L_a = L_a\Theta_b$.
		\item $\Theta_{a\mult b} =\Theta_{a\mult^\prime b} = \Theta_a\Theta_b=\Theta_b\Theta_a$.
		\item $\Theta_a^{-1} = \Theta_{a^{-1}}$.
	\end{enumerate}
\end{proposition}
\begin{proof}
	By invariance of $\mult$ and proposition \ref{prop:automorphismcommute}: $\Theta_bL_a=L_{\Theta_b(a)}\Theta_b = L_a\Theta_b$. Of course $L_{a\mult^\prime b}^\prime = L_{a\mult^\prime b} \Theta_{a\mult^\prime b}$ and also $L_{a\mult^\prime b}^\prime = L_a^\prime L_b^\prime = L_a\Theta_a L_b\Theta_b = L_aL_b \Theta_a\Theta_b = L_{a\mult b}\Theta_a\Theta_b$. Because $a\mult^\prime b= a\mult b$ by proposition \ref{prop:commutingeffects} and since $a$ and $b$ are both invertible we conclude that $L_{a\mult^\prime b}=L_{a\mult b}$ is an invertible operator so that we must have $\Theta_{a\mult b} = \Theta_a\Theta_b$. Since $\Theta_{a\mult b}=\Theta_{b\mult a}$ we also see that $\Theta_a\Theta_b =\Theta_b\Theta_a$. Finally because $a$ and $a^{-1}$ commute and $\Theta_1=\id$ we can immediately conclude the third point.
\end{proof}

\subsection{Invariance characterises the sequential product}
In this section we will use the invariance property $\Theta(a\mult b)=\Theta(a)\mult \Theta(b)$ to characterise the standard sequential product on $B(H)$ where $H$ is a (not necessarily finite-dimensional) complex Hilbert space. We will then give a version of this theorem that holds for arbitrary finite-dimensional von Neumann algebras and finally we generalise the theorem to hold for (also not necessarily finite-dimensional) Euclidean Jordan algebras. We note that the standard sequential product on these spaces is invariant so that proposition \ref{prop:isomorphismgroup} holds.

\begin{lemma}\label{lem:commutingunitary}
	Let $\mult^\prime$ be a sequential product on $B(H)$ and let $a\in B(H)$ be a simple invertible positive element, then there exists a unitary $u_a$ in $C(a)$, the double commutant of $a$, such that $a\mult^\prime b = \sqrt{a}u_a^*bu_a\sqrt{a}$ for all $b$.
\end{lemma}
\begin{proof}
	Because $a$ is invertible, it is order preserving so that $a\mult^\prime b = a\mult \Theta_a(b)$ for some unital order isomorphism $\Theta_a$ by proposition \ref{prop:quotient}. On $B(H)^{sa}$ the unital order isomorphisms are precisely the Jordan isomorphisms. Such an isomorphism is either a $*$-automorphism or a $*$-antiautomorphism. Because $\Theta_a=\Theta_{\sqrt{a}^2} = \Theta_{\sqrt{a}}^2$ by proposition \ref{prop:isomorphismgroup} it can not be an antiautomorphism. On $B(H)$, all automorphisms are \emph{inner}, so that $\Theta_a(b) = u_a^* b u_a$ for some unitary $u_a$. By proposition \ref{prop:automorphismcommute} we see that whenever $b$ commutes with $a$ that $\Theta_a(b)=b$. It follows that then $u_ab=bu_a$ for all $b$ commuting with $a$ (in a von Neumann algebra $a\mult b=b\mult a$ if and only if $ab=ba$ \cite{gudder2002sequentially}). We conclude that $u_a$ is indeed in the bicommutant of $a$.
\end{proof}

\begin{theorem}\label{theor:uniqueinvariantproduct}
	If $\mult^\prime$ is an invariant sequential product on $B(H)$ that is norm-continuous in its first coordinate, then $a\mult^\prime b = \sqrt{a}b\sqrt{a}$ for all effects $a$ and $b$.
\end{theorem}
\begin{proof}
	Since we assume norm-continuity in the first coordinate, it suffices if we prove it for a norm-dense set, e.g. for all simple invertible positive $a$. We will first prove that if $a$ is diagonal in the standard basis that $a\mult^\prime b = \sqrt{a}b\sqrt{a}$. For an arbitrary $a\geq 0$  we can then use its spectral decomposition to write $a=u^*du$ where $d$ is diagonal. Letting $\Phi$ denote conjugation with $u$ we see that $\Phi(a\mult^\prime b) = d\mult^\prime \Phi(b) = d\mult \Phi(b) = \Phi(a\mult b)$ using invariance so that indeed $a\mult^\prime b = a\mult b$.

	So let $a$ be a simple invertible diagonal positive effect. Let $u_a$ denote the unitary from the previous lemma. Since $u_a$ is in the bicommutant of $a$, it will also be diagonal. Let $b^T$ denote the transpose of an element $b$, then because $a$ and $u_a$ are diagonal we get $a^T=a$, $u_a^T=u_a$ and $(u_a^*)^T=u_a^*$. The transpose map is a unital order isomorphism so we must have $(a\mult^\prime b)^T = a^T\mult^\prime b^T = a\mult^\prime b^T = \sqrt{a}u_a^*b^Tu_a\sqrt{a}$. But also $(a\mult^\prime b)^T = (\sqrt{a}u_a^*bu_a\sqrt{a})^T = \sqrt{a}^Tu_a^T b^T (u_a^*)^T\sqrt{a}^T = \sqrt{a}u_ab^T u_a^*\sqrt{a}$. The square root of $a$ is invertible since $a$ is allowing us to conclude that $u_a^* b^T u_a = u_a b^T u_a^*$ for all $b$. This is only possible when $u_a^2$ is in the centre of the algebra. But then of course the conjugation with this element is trivial so that $\Theta_a^2 = \id$. Replacing $a$ by $\sqrt{a}$ and noting that $\Theta_{\sqrt{a}}^2 = \Theta_{\sqrt{a}^2} =\Theta_a$ by proposition \ref{prop:isomorphismgroup} we see that $\Theta_a=\id$ whenever $a$ is simple, invertible and diagonal.
\end{proof}

As we will show below, this theorem also holds when considering direct sums of $B(H)$'s. It is however not clear if it holds for arbitrary von Neumann algebras. The proof relies on the existence of a suitable order automorphism (namely the transpose), and in general this automorphism doesn't exist.

\begin{corollary}
	Any invariant sequential product on a finite-dimensional von Neumann algebra that is norm-continuous in its first argument must be the standard sequential product.
\end{corollary}
\begin{proof}
	As in the previous theorem it suffices to consider simple invertible positive effects $a$. Let $\Theta_a$ denote the unital order isomorphism related to the sequential product. This is still a $*$-automorphism by the same argument as in the above theorem. A finite-dimensional von Neumann algebra is a direct sum of $B(H)$'s and any $*$-automorphism then consists of a composition of a map interchanging isomorphic factors followed by some unitary conjugations. Let $k$ denote the amount of factors in the algebra. Since $\Theta_a = \Theta_{a^{1/(k!)}}^{k!}$ we see that $\Theta_a$ cannot contain any interchanges of factors and must be a direct sum of unitary conjugations of each of the factors. Then restricting to each factor separately and using the above theorem we see that indeed the sequential product must be the standard one.
\end{proof}
\begin{theorem}\label{theor:uniqueinvariantproductEJA}
	If $\mult^\prime$ is an invariant sequential product on a Euclidean Jordan algebra that is norm-continuous in its first argument, then $a\mult^\prime b = Q_{\sqrt{a}} b$ for all effects $a$ and $b$.
\end{theorem}
\begin{proof}
	Due to the same argument as in the previous corollary we only need to consider simple EJA's. By referring to the Jordan-von Neumann-Wigner classification of EJA's \cite{jordan1993algebraic} these are either $B(H)^\sa$ where $H$ is a real, complex, quaternionic or octonion Hilbert space, or a type of algebra called a \emph{spin-factor}. It is straightforward to check that the proof of theorem \ref{theor:uniqueinvariantproduct} goes through for a $B(H)$ where the $H$ is a real, complex, quaternion or octonion Hilbert space. The only case left to check then are the spin factors.

	A spin-factor is a space $V=H\oplus \R$ where $H$ is a real Hilbert space. Any unital order isomorphism on this space is implemented by an orthogonal operator on $H$. We therefore have $(v,t)\mult^\prime (w,s) = (v,t)\mult (U_{v,t}w, s)$ where $U_{v,t}$ is an orthogonal operator such that $U_{v,t}v = v$ when $(v,t)$ is invertible due to proposition \ref{prop:automorphismcommute}. It is a straightforward verification that the invariance property of the sequential product by order isomorphisms means that for any orthogonal matrix $W$ we must have $WU_{v,t} = U_{Wv,t}W$, so that for any $W$ with $Wv=v$ we get $WU_{v,t} = U_{v,t}W$ implying that $U_{v,t}$ is central in the subspace of $H$ given by $\{v\}^\perp$. The only central elements are the identity or minus the identity and since the operator must be a square it must be the identity so that the proof is completed.
\end{proof}


\subsection{Symmetry characterises the sequential product}
In this section we will consider sequential effect spaces that are also Hilbert spaces as in definition \ref{def:HilbertSES}. As was shown in theorem \ref{theor:hilbertisEJA} such sequential effect spaces are Euclidean Jordan algebras. In this section we will show that the property of $\inn{a\mult b, c} = \inn{b,a\mult c}$ actually fixes the sequential product.

\begin{lemma}\label{lem:orthfinite}
	Let $E$ be a Hilbert SES, then any set of non-zero orthogonal sharp effects is finite.
\end{lemma}
\begin{proof}
	Let $\norm{a}_2 = \sqrt{\inn{a,a}}$ denote the norm of the inner product in which the space is complete by assumption and towards contradiction suppose we have an infinite set of non-zero orthogonal sharp effects $p_i$. Note that $\inn{p_i,p_j}=\inn{p_i\mult p_i,p_j}=\inn{p_i,p_i\mult p_j}=\inn{p_i,0}=0$ when $i\neq j$ so that they are also orthogonal with respect to the inner product.

	Since $\norm{\sum_i p_i}_2\leq \norm{1}_2<\infty$ we have $\sum_i \inn{p_i,p_i}< \infty$, so that $(\inn{p_i,p_i})_i$ is a sequence converging to zero. There must then be a subsequence $p_{n_i}$ such that $\inn{p_{n_i},p_{n_i}}\leq 2^{-i}/i^2$. Now let $q_k = \sum_{l=1}^k l p_{n_l}$ as an element of the order unit space associated to $E$, then $\norm{q_k - q_m}_2 = \sqrt{\sum_{l=m}^k l^2 \inn{p_{n_l},p_{n_l}}} \leq \sqrt{\sum_{l=m}^k 2^{-l}}$ so that $(q_k)$ is a Cauchy sequence. Because the space is complete in $\norm{\cdot}_2$ it must then have a limit $q$ and since $(q_k)$ is an increasing sequence this will also be the supremum so that $q\geq q_k$ for all $k$. In the regular norm we get $\norm{q}\geq \norm{q_k}=k$ for all $k$ which is not possible. We conclude that any set of orthogonal idempotents must indeed be finite.
\end{proof}
\begin{corollary}
	Let $E$ be a Hilbert SES, and let $V$ be its associated order unit space. Let $a\in V$. Then there exists a finite set of orthogonal sharp effects $p_i\in E$ and $\lambda_i\in \R$ such that $a=\sum_i \lambda_i p_i$.
\end{corollary}
\begin{proof}
	For an $a\in V$ we can look at the bicommutant $C(a)$ as in the proof of theorem \ref{theor:spectraltheorem}. If $C(a)$ is infinite-dimensional then there is also an infinite set of orthogonal sharp effects contradicting the previous lemma. So $C(a)$ is finite-dimensional which means we can write $a=\sum_{i=1}^n \lambda_i p_i$ for some finite $n\in \N$.
\end{proof}

\begin{theorem}\label{theor:uniquesymmetricproduct}
	Let $E$ be a sequential effect space that is also a Hilbert space with respect to an inner product $\inn{\cdot,\cdot}$. If the sequential product is symmetric with respect to the inner product, then $E$ is the unit interval of a Euclidean Jordan algebra and the sequential product is identical to its standard sequential product.
\end{theorem}
\begin{proof}
	By theorem \ref{theor:hilbertisEJA} we know that $E$ is the unit interval of a Euclidean Jordan algebra so that its associated OUS $V$ is an EJA. Of course the standard sequential product is symmetric with respect to the inner product. We denote this product by $\mult$ and the product maps by $L_a$. We let the given sequential product be denoted by $\mult^\prime$ and the product maps by $L_a^\prime$.

	First of all, when $p$ is sharp we get $\inn{p\mult^\prime a, b} = \inn{p\mult(p\mult^\prime a),b} = \inn{p\mult^\prime a, p\mult b} = \inn{a, p\mult^\prime(p\mult b)} = \inn{a,p\mult b} = \inn{p\mult a,b}$ where we have used that $p\mult a\leq p$ and when $c\leq p$ we get $p\mult c = c$ (and the same for $\mult^\prime$). We conclude that $L^\prime_p = L_p$ when $p$ is sharp. As a result we have for an arbitrary effect $a$ the equality $L_a^\prime = L_a^\prime L_{\ceil{a}}$ so that we can restrict $V$ to the order ideal generated by $\ceil{a}$. It therefore will suffice to consider effects $a$ with $\ceil{a}=1$. By the previous corollary all elements are simple so that $\ceil{a}=1$ if and only if $a$ is invertible. By proposition \ref{prop:finiteranksequentialorder} we can then conclude that any sequential product sequentially orders the space so that $\mult^\prime$ is related to $\mult$ by unital order isomorphisms: $L_a^\prime = L_a\Theta_a = \Theta_aL_a$ (proposition \ref{prop:isomorphismgroup}). We will show that $\Theta_a^2 = \id$, which completes the proof because every positive invertible element is a square.

	First, the $\Theta_a$ are symmetric with respect to the inner product:
	\begin{align*}
		\inn{\Theta_a(b),c} &= \inn{L_{a^{-1}}L_a\Theta_a(b),c}=\inn{L_a^\prime(b),L_{a^{-1}}(c)} = \inn{b, L_a^\prime L_{a^{-1}}(c)} \\
		&= \inn{b, L_a \Theta_a L_{a^{-1}}(c)} = \inn{b,L_aL_{a^{-1}}\Theta_a(c)} = \inn{b,\Theta_a(c)}
	\end{align*}
	where we have used respectively that $L_{a^{-1}} = L_a^{-1}$, that $L_{a^{-1}}$ and $L_a^\prime$ are symmetric with respect to the inner product, that $L_a^\prime = L_a\Theta_a$ and that $\Theta_a$ commutes with $L_{a^{-1}}$. Now let $b$ and $c$ be arbitrary and note that $\inn{\Theta_a^2(b),c} = \inn{\Theta_a(b),\Theta_a(c)} = \inn{1,\Theta_a(b)\mult \Theta_a(c)} = \inn{1,\Theta_a(b\mult c)} = \inn{\Theta_a(1),b\mult c} = \inn{1,b\mult c} = \inn{b,c}$. Since $\Theta_a^2$ fixes the entire inner product we indeed have $\Theta_a^2 = \id$ and the proof is finished.
\end{proof}
\noindent \textbf{Note}: This characterisation is similar to the one given in \cite{gudder2008characterization}, but it is more general in that it applies to arbitrary Euclidean Jordan algebras, and not just to the effects on a Hilbert space. We also don't have any continuity requirement.

\subsection{Invertibility preservation characterises the sequential product}
The Jordan product and the quadratic representation of a Jordan algebra are very closely related. In fact, in the theory of \emph{quadratic Jordan algebras} the quadratic representation is taken as fundamental \cite{mccrimmon1966general} (and over the field of real numbers, quadratic Jordan algebras actually coincide with regular Jordan algebras). An essential part of the definition of a quadratic Jordan algebra is the following equality.

\begin{definition}
	The \emph{fundamental equality} for the quadratic representation of a Jordan algebra is the equation $Q_{Q_ab} = Q_aQ_bQ_a$ which holds for every element $a$ and $b$.
\end{definition}
Recall that the quadratic representation for a von Neumann algebra is given by $Q(a)b=aba$. In this case the fundamental equality says that for any $a$, $b$ and $c$ we have $(aba)c(aba) = a(b(aca)b)a$, a statement that is obviously true.

For the sequential product map $L_a:=Q_{\sqrt{a}}$ of a positive $a$ the fundamental equality translates into $L_{(a\mult b)^2} = L_aL_{b^2}L_a$. This property implies an interesting result regarding the inverses of effects.

\begin{definition}
	We call a sequential product $\mult$ with product maps $L_a$ \emph{quadratic} if for all effects $a$ and $b$ we have $L_{(a\mult b)^2} = L_aL_{b^2}L_a$. We say that $\mult$ is \emph{invertibility preserving} when for all positive invertible $a$ and $b$ their sequential product is invertible and $(a\mult b)^{-1}= a^{-1}\mult b^{-1}$.
\end{definition}

\begin{proposition}
	Let $E$ be a SES with a sequential product $\mult$. If $\mult$ is quadratic, then $\mult$ is invertibility preserving.
\end{proposition}
\begin{proof}
	Let $a$ and $b$ be positive invertible effects.
	$L_{(a^{-1}\mult b^{-1})^2}L_{(a\mult b)^2} = L_{a^{-1}}L_{b^{-2}}L_{a^{-1}}L_aL_{b^2}L_a = \id$ and similarly $L_{(a\mult b)^2}L_{(a^{-1}\mult b^{-1})^2}=\id$. Inserting $1$ into the expressions then shows that $(a^{-1}\mult b^{-1})^2=(a\mult b)^{-2}$. Taking square roots then gives the desired expression.
\end{proof}

This condition of invertibility preservation can also easily be seen to be true in a von Neumann algebra: $(a\mult b)^{-1} = (\sqrt{a}b\sqrt{a})^{-1} = \sqrt{a}^{-1}b^{-1}\sqrt{a}^{-1} = \sqrt{a^{-1}}b^{-1}\sqrt{a^{-1}} = a^{-1}\mult b^{-1}$.

\begin{proposition}
	Let $E$ be a SES with associated order unit space $V$ and denote the set of positive invertible elements of $V$ by $V_{>0}$. Suppose $E$ has a sequential product that is invertibility preserving, then the inverse map $\Phi: V_{>0}\rightarrow V_{>0}$ given by $\Phi(a):=a^{-1}$ is an order antiautomorphism: $a\leq b \iff b^{-1}\leq a^{-1}$.
\end{proposition}
\begin{proof}
	Let $a$ and $b$ be positive and invertible and suppose $a^{-1}\leq b$. Because $L_{b^{-1}}$ is an order automorphism we get $a^{-1}\leq b \iff b^{-1}\mult a^{-1}\leq b^{-1}\mult b = 1$ and then using invertibility preservation: $b^{-1}\mult a^{-1}=(b\mult a)^{-1}\leq 1\iff 1\leq b\mult a \iff b^{-1}\leq b^{-1}\mult (b\mult a) = a$.
\end{proof}

In theorem \ref{theor:uniqueinvariantproduct} we gave a characterisation of the standard sequential product as the unique one satisfying $\Phi(a\mult b)=\Phi(a)\mult\Phi(b)$ for all unital order automorphisms. The above proposition shows that any SES that allows an invertibility preserving sequential product (such as a von Neumann algebra) has $a\leq b \iff b^{-1}\leq a^{-1}$ so that the inverse map is a unital order \emph{anti}automorphism. The following characterisation can therefore be seen as complementary to theorem \ref{theor:uniqueinvariantproduct}.

\begin{theorem}\label{theor:uniquequadraticproduct}
	Let $E$ be a SES with an invariant, norm-continuous and invertibility preserving sequential product $\mult$. If $\mult^\prime$ is a norm-continuous and invertibility preserving sequential product, then $\mult=\mult^\prime$.
\end{theorem}
\begin{proof}
	Let $a,b\in E_0$ be simple and invertible.
	We note that $a\mult^\prime b = (a\mult \Theta_a(b)) = \Theta_a (a\mult b)$ by proposition \ref{prop:isomorphismgroup} and similarly $a^{-1}\mult^\prime b^{-1} = \Theta_{a^{-1}}(a^{-1}\mult b^{-1})$. We also have $1=\Theta_a(1)=\Theta_a(c\mult c^{-1}) = \Theta_a(c)\mult \Theta_a(c^{-1})$ so that $\Theta_a(c)^{-1}=\Theta_a(c^{-1})$ for all invertible $c$.

	Now $a^{-1}\mult (\Theta_{a^{-1}}(b^{-1})) = a^{-1}\mult^\prime b^{-1} = (a\mult^\prime b)^{-1} = (a\mult \Theta_a(b))^{-1}  =  \Theta_a (a\mult b)^{-1} = \Theta_a((a\mult b)^{-1}) = \Theta_a(a^{-1}\mult b^{-1}) = a^{-1}\mult (\Theta_a (b^{-1}))$. Applying $L_a$ to the left on both sides we get $\Theta_a^{-1}(b^{-1}) = \Theta_a(b^{-1})$ from which we conclude that $\Theta_a^2(b^{-1}) = b^{-1}$ for all simple invertible $a$ and $b$. Since the $\Theta_a$ are continuous and the linear span of the simple invertible effects lies norm-dense in the space we conclude that $\Theta_a^2 = \Theta_{a^2} = \id$. All invertible effects are squares so that $a\mult^\prime b = a\mult b$ for all simple invertible $a$. Because the sequential product is assumed to be norm-continuous in the first argument and the simple invertible effects lie dense in $E$ we are done.
\end{proof}
\noindent\textbf{Note}: We only need the invariance of $\mult$ to make sure proposition \ref{prop:isomorphismgroup} holds. It is not clear whether this condition is necessary. The conditions stipulated in the above theorem are in any case satisfied by the unit interval of any $C^*$-algebra in which the linear span of the projections lie norm-dense in the algebra. It is satisfied in particular by von Neumann algebras or more generally AW$^*$-algebras. The conditions are also satisfied by Euclidean Jordan algebras or, more generally, by JBW-algebras, their infinite-dimensional counterpart \cite{alfsen2012geometry}. 

\section{Conclusion}
We have studied convex $\sigma$-sequential effect algebras and have shown that they have a structure similar to that of operator algebras considered in quantum theory. We have used our results to give three new characterisations of the sequential product. While the scope of each of these characterisations is slightly different, they all hold at least for effects on a Hilbert space, and for Euclidean Jordan algebras.

A couple of questions remain. 
\begin{itemize}
	\item By theorem \ref{theor:hilbertisEJA} an inner product is enough to conclude that convex $\sigma$-sequential effect algebras are Jordan algebras. Is this additional assumption of an inner product necessary? Or, in other words, do there exist convex $\sigma$-sequential effect algebras that are not Jordan algebras?
	\item For our results regarding uniqueness of the sequential product we needed to assume the existence of an invariant sequential product. Is there a convex sequential effect algebra that does not have an invariant sequential product?
	\item Regarding the scope and conditions of our uniqueness theorems, does the characterisation by invariance of theorem \ref{theor:uniqueinvariantproduct} hold for an arbitrary von Neumann algebra (instead of just for factor I algebras) and how essential is the condition of norm-continuity in theorems \ref{theor:uniqueinvariantproduct} and \ref{theor:uniquequadraticproduct}?
\end{itemize}

It is interesting to note that the three characterisations are quite different in nature, being based on respectively preservation of the product by order isomorphisms, symmetry with respect to an inner product, and an algebraic property, while still having some striking resemblances in each of their proofs. In particular, all the proofs rely on showing that some order isomorphism must square to the identity and then showing that these isomorphisms must already have been a square. This same proof step also occurs in the characterisations of \cite{gudder2008characterization} and \cite{westerbaan2016universal} and is the reason the authors of these papers include the axiom that $a\mult (a\mult b) = (a\mult a)\mult b$. It might be interesting to see whether a characterisation of the sequential product can be found that does not include a variation on this axiom.

\textbf{Acknowledgements}: The author would like to thank Bas and Bram Westerbaan for all the useful and insightful conversations regarding effect algebras and order unit spaces. This work is supported by the ERC under the European Union’s Seventh Framework Programme (FP7/2007-2013) / ERC grant n$^\text{o}$ 320571.

\bibliographystyle{eptcs}
\bibliography{../bibliography}

\appendix

\section{Euclidean Jordan algebras have a sequential product}
Let $V$ be a Euclidean Jordan algebra with $a,b\in V$ positive and let $a\mult b = Q_{\sqrt{a}} b$. Without further reference we will use that $Q_a^2 = Q_{a^2}$ and that $Q_aQ_bQ_a = Q_{Q_a b}$ (the \emph{fundamental equality}). 
\begin{itemize}
	\item Obviously $1\mult a = a$. 
	\item By \cite[Lemma 1.26]{alfsen2012geometry} $a\mult b = 0 \iff b\mult a =0$.
	\item By \cite[Proposition 2.4]{alfsen2012geometry} the map $Q_a$ is normal for all $a$ so that if $a\mult b$ is a sequential product, then it will indeed preserve infima.
\end{itemize}
\begin{proposition}
	Effects $a$ and $b$ commute ($a\mult b = b\mult a$) if and only if their quadratic representations commute: $[Q_a,Q_b]=Q_aQ_b-Q_bQ_a=0$.
\end{proposition}
\begin{proof}
	An EJA is a Hilbert space, and $Q_a$ and $Q_b$ are bounded positive operations (\cite[Theorem 1.25]{alfsen2012geometry}) so $Q_a, Q_b \in B(V)$ and $Q_a,Q_b\geq 0$.
	When $[Q_a,Q_b]=0$ we then also have $[Q_{\sqrt{a}},Q_{\sqrt{b}}]=0$ so that indeed $a\mult b = Q_{\sqrt{a}}b = Q_{\sqrt{a}}Q_{\sqrt{b}}1 = Q_{\sqrt{b}}Q_{\sqrt{a}}1 = b\mult a$.

	 For the other direction suppose we have $a^2\mult b^2 = b^2\mult a^2$. This holds precisely when $Q_a b^2 = Q_b a^2$. By the fundamental equality we then have $Q_aQ_b^2Q_a = Q_{Q_a b^2} = Q_{Q_b a^2} = Q_bQ_a^2Q_b$. By \cite[Theorem 2]{gudder2002sequentially} we see that when $Q_aQ_b^2Q_a = Q_bQ_a^2Q_b$ then $Q_aQ_b^2 = Q_b^2 Q_a$ and $Q_bQ_a^2 = Q_a^2 Q_b$. We conclude that $Q_a^2$ and $Q_b$ commute so $Q_a^2=Q_{a^2}$ also commutes with $Q_b^2=Q_{b^2}$. Replacing $a^2$ and $b^2$ with $a$ and $b$ gives the desired conclusion.
\end{proof}
\begin{corollary}
	If $a$ and $b$ commute, then $a\mult(b\mult c) = (a\mult b)\mult c$ and if additionally $a$ commutes with $c$, then $a\commu (b\mult c)$.
\end{corollary}
\begin{proof}
	If $a$ and $b$ commute then their quadratic representations commute so we calculate $(Q_a b)^2 = Q_{Q_a b} 1 = Q_aQ_bQ_a 1 = Q_{a^2} Q_b 1 = Q_{a^2} b^2$. By taking square roots we also have $\sqrt{Q_a b} = Q_{\sqrt{a}} \sqrt{b}$. We use this in the following calculation to get the desired result.
	$$a\mult(b\mult c) = Q_{\sqrt{a}} Q_{\sqrt{b}} c = Q_{a^{1/4}} Q_{\sqrt{b}}Q_{a^{1/4}}\,c = Q_{Q_{a^{1/4}} \sqrt{b}} c = Q_{\sqrt{Q_{\sqrt{a}} b}} c = (a\mult b)\mult c.$$

	If $a\commu b$ and $a\commu c$, then $Q_aQ_b = Q_bQ_a$ and $Q_aQ_c =Q_cQ_a$, so that also $Q_{Q_bc} Q_a = Q_bQ_cQ_b Q_a = Q_a Q_bQ_cQ_b = Q_a Q_{Q_bc}$ so that indeed $a\commu (b\mult c)$.
\end{proof}

It seems very unlikely that the following proposition is not already known, but nevertheless the author couldn't find any reference of it in the literature.
\begin{proposition}
	Let $T_a(b) = a*b$ denote the Jordan multiplication of $a$. For effects $a$ and $b$ we have $[Q_a,Q_b]=0$ if and only if $[T_a,T_b]=0$.
\end{proposition}
\begin{proof}
	Recall that $Q_a = 2T_a^2 - T_{a^2}$ so if $[T_a,T_b]=0$ we also have $[Q_a,Q_b]=0$.

	So for the other direction lets suppose $[Q_a,Q_b]=0$. It suffices to consider invertible $a$ and $b$ because otherwise we note that $Q_aQ_b = Q_{\ceil{a}}Q_aQ_b = Q_{\ceil{a}}Q_bQ_a = Q_{\ceil{a}}Q_b Q_a Q_{\ceil{a}}$ so that we can restrict to the sub-algebra where they are invertible. By \cite[Proposition II.3.4]{faraut1994analysis} we have $Q_{\exp a} = \exp(2T(a))$. By the invertibility of $a$, its logarithm $\ln a$ can be defined using the spectral theorem. We write $Q_{a} = Q_{\exp(\ln(a))} = \exp(2T(\ln(a)))$ so that $[\exp(T(\ln(a))), Q_b]=0$. The exponent of an operator lies in its bicommutant allowing us to conclude that also $[T(\ln(a)),Q_b]=0$ and by the same reasoning but with $b$ instead of $a$ we get $[T(\ln a), T(\ln b)]=0$. Reasoning inside of $V$ instead of $B(V)$ we can look at the Jordan algebra spanned by $\ln a$ and $\ln b$. This is a commutative associative algebra so that by \cite[Proposition 1.12]{alfsen2012geometry} we have a spectral theorem that allows us to define $\exp\ln a = a$ and $\exp\ln b = b$ that also still operator commute: $[T(a),T(b)]=0$.
\end{proof}
\begin{corollary}
	If $[Q_a,Q_b]=0$ then $Q_ab = T_{a^2}b = a^2*b$.
\end{corollary}
\begin{proof}
	Suppose $[Q_a,Q_b]=0$, then by the previous proposition $[T_a,T_b]=0$. We then calculate
	$$Q_a b = 2 T_a^2 b - T_{a^2}b = 2T_a^2 T_b 1 - T_{a^2} T_b 1 = 2T_b T_a^2 1 - T_b T_{a^2}1 = T_b (a^2) = b*a^2 = a^2*b.$$
\end{proof}
\begin{corollary}
	If $a\commu b$ then $a\commu b^\perp$ and if also $a\commu c$ then $a\commu b+c$.
\end{corollary}
\begin{proof}
	We have seen that $a\commu b$ implies that $[Q_a,Q_b]=0$ which in turn implies that $[T_a,T_b]=0$. It then immediately follows that $[T_a,T_{b^\perp}]=0$ so that $a\mult b^\perp = a*b^\perp = b^\perp*a = b^\perp \mult a$ by the previous corollary. The same trick can be used to conclude that $a\commu (b+c)$ when $a\commu c$.
\end{proof}

\begin{theorem}
	Let $V$ be a Euclidean Jordan algebra and let $a\mult b = Q_{\sqrt{a}}b$, for positive $a$ and $b$, then the unit interval of $V$ is a convex $\sigma$-sequential effect algebra.
\end{theorem}
\begin{proof}
	All the necessary properties have been shown in the previous propositions and corollaries.
\end{proof}

\section{Results concerning basically disconnected spaces}
\begin{proposition}
	Let $X$ be basically disconnected, then the set of clopens forms an $\omega$-complete lattice, that is, it has countable joins and meets.
\end{proposition}
\begin{proof}
	We first construct the join.
	Let $\{S_i\}$ be some countable collection of clopens. Let $S=\bigcap S_i$. Since all the $S_i$ are open, $S$ is a $G_\delta$ set and because all the $S_i$ are closed, $S$ is also closed. The set $S$ is then compact because it is a closed subset of the compact $X$. In a regular space (which any compact Haussdorff space is) any compact $G_\delta$ set is a zero-set, and because $X$ is basically disconnected, we see that the interior of $S$ must then be clopen. The claim is that $S^\circ$ is the greatest lower bound of the $S_i$. That it is a lower bound follows by construction. Suppose there is some $U$ that is also a lower bound, then $U$ must be clopen and satisfy $U\subseteq \bigcap S_i = S$. Taking the interior on both sides and using that $U$ is clopen we see immediately that indeed $U\subseteq S^\circ$.

	By taking the De Morgan dual it can be shown that the meet of $\{S_i\}$ is given by $\cl{\bigcup S_i}$
\end{proof}

\begin{proposition}
	If $X$ is basically disconnected, then $C(X)$ is $\sigma$-Dedekind complete: that is, any countable bounded set of continuous functions has a supremum.
\end{proposition}
\begin{proof}
	We adapt the proof from the lecture notes of A. van Rooij (2011) on Riesz spaces.

	Let $\{f_i\}$ be a collection of continuous maps $f_i:X\rightarrow \R^+$ (we may without loss of generality assume that they are positive) that have some continuous upper bound $f_0$. Define $g(x):= \sup_i f_i(x)$. Now let $s>0$, then $g^{-1}((s,\infty)) = \bigcup_i f^{-1}((s,\infty))$ is open. Let $U_s=\cl{\bigcup_i f^{-1}((s,\infty))} = \cl{\bigcup_i \cl{f^{-1}((s,\infty))}}$ be a closure of a countable union of clopens so that by the previous proposition it is again a clopen set. By construction $f_0\geq s\xi_{U_s}$. Define $f_\infty(x) := \sup_s s\xi_{U_s}(x) = \sup\{s~;~ x\in U_s\}$. Again by construction $g\leq f_\infty\leq f_0$ so if we show that $f_\infty$ is continuous it is an upperbound, and since $f_0$ was arbitrary it will also be the least upper bound. We note that $f_\infty^{-1}((t,\infty)) = \bigcup_{s>t} U_s$ which is a union of open sets, so is open. And since $f_\infty(x)<s \implies x\not\in U_s\implies f_\infty(x)\leq s$ we get $f_\infty^{-1}([0,t)) = \bigcup_{s<t}X\backslash U_s$ which is again a union of open sets so is open. Since the inverse images of $f_\infty$ on a basis of $\R^+$ are open, it is indeed continuous.
\end{proof}

\begin{proposition}
	Let $E$ be a commutative SES, then any $a\in E$ is the supremum and norm limit of an increasing sequence of effects of the form $\sum_{k=1}^n \lambda_i p_i$ where $\lambda_i>0$ and the $p_i$ are orthogonal sharp effects.
\end{proposition}
\begin{proof}
	The effect $E$ corresponds to the unit interval of a $C(X)$ where $X$ is a basically disconnected compact Haussdorff space, so an effect $a$ can be represented by a continuous function $a:X\rightarrow [0,1]$. The sharp effects correspond to clopens in $X$. Write $p_n^k$ for the characteristic function of the clopen set $\cl{a^{-1}((\frac{k}{n},1])}$, so that we have the implication $p_n^k(x)=1 \implies a(x)\geq \frac{k}{n}$. Of course $p_n^k\geq p_n^{k+1}$. Define $q_n:= \sum_{k=1}^n \frac{1}{n}p_n^k$, then $q_n$ only takes values $\frac{l}{n}$ for some $0\leq l\leq n$ and we obviously have $q_n(x)=\frac{l}{n} \iff p_n^l(x)=1 \implies a(x)\geq \frac{l}{n}$ so that $q_n\leq a$.

	In general $q_n$ and $q_m$ might not have an obvious order relation, but if we consider $q_n$ and $q_{2n}$ then it is straightforward to check that the latter will always be larger. The sequence $(q_{2^i})$ is therefore an increasing sequence that has $a$ as an upper bound. We will show that $\norm{a-q_{2^n}}\leq 2^{1-n}$ so that this sequence indeed converges to $a$. As it is an increasing sequence, $a$ will then also be its supremum.

	For every $x\in X$ there exists a $0\leq l\leq 2^n$ such that $\frac{l}{2^n}\leq a(x)< \frac{l+1}{2^n}$. We then always have either $q_{2^n}(x)=\frac{l-1}{2^n}$ or $q_{2^n}(x)=\frac{l}{2^n}$. In both cases $a(x)-q_{2^n}(x) \leq \frac{2}{2^n}=2^{1-n}$. Since this bound does not depend on $x$ or $l$, we indeed have $\norm{a-q_{2^n}} \leq 2^{1-n}$.

	The last thing we have to show is that we can write $q_n$ as a linear combination of orthogonal sharp effects, but this is easily done by writing $q_n = \sum_{k=1}^n \frac{1}{n} p_n^k$ as $q_n = \sum_{k=1}^n \frac{k}{n} (p_n^k - p_n^{k+1})$.
\end{proof}

\end{document}

%% file: preamble.tex
\usepackage[utf8]{inputenc}
\usepackage[english]{babel}
\usepackage{amsmath}
\usepackage{amsfonts}
\usepackage{amsthm}
\usepackage{graphicx}
\usepackage{physics}
\usepackage{mathtools}
\usepackage{xspace}
\usepackage{tikz-cd}
\usepackage{hyperref}
\usepackage[toc,page]{appendix}

\usepackage{enumerate}

\DeclarePairedDelimiter{\ceil}{\lceil}{\rceil}
\DeclarePairedDelimiter{\floor}{\lfloor}{\rfloor}
\DeclarePairedDelimiter{\inn}{\langle}{\rangle}

\theoremstyle{definition}
\newcounter{counter}
\numberwithin{counter}{section}
\newtheorem{theorem}[counter]{Theorem}
\newtheorem{proposition}[counter]{Proposition}
\newtheorem{definition}[counter]{Definition}
\newtheorem{lemma}[counter]{Lemma}
\newtheorem{corollary}[counter]{Corollary}
\newtheorem{example}[counter]{Example}
\newtheorem*{theorem*}{Theorem}
\newtheorem*{lemma*}{Lemma}

\newcommand{\R}{\mathbb{R}}

\newcommand{\N}{\mathbb{N}}

\newcommand{\cl}[1]{\overline{#1}}

\newcommand{\id}{\text{id}}

\newcommand{\sa}{\text{sa}}